\documentclass[11pt,reqno]{amsart}
\usepackage{amssymb,amsfonts,amscd,amsbsy}
\usepackage[mathscr]{eucal}
\usepackage{url}

\usepackage{amsmath}
\usepackage{color}
\usepackage{tabu}

\newtheorem{theorem}{Theorem}[section]
\newtheorem{lemma}[theorem]{Lemma}
\newtheorem{proposition}[theorem]{Proposition}
\newtheorem{corollary}[theorem]{Corollary}
\newtheorem{conjecture}[theorem]{Conjecture}
\theoremstyle{definition}
\newtheorem{remark}[theorem]{Remark}
\newtheorem{example}[theorem]{Example}
\numberwithin{equation}{section}

\newcommand{\ph}[2]{{\left({#1}\right)}_{#2}}
\newcommand{\bin}[2]{\left({\genfrac{}{}{0pt}{}{#1}{#2}}\right)}
\newcommand{\fac}[1]{{#1}\text{!}}

\newcommand{\assign}{:=}
\renewcommand*{\bar}{\overline}

\makeatletter
\def\imod#1{\allowbreak\mkern5mu({\operator@font mod}\,\,#1)}
\makeatother

\begin{document}

\title[Sequences, modular forms and cellular integrals]{Sequences, modular forms and cellular integrals}

\author{Dermot McCarthy, Robert Osburn, Armin Straub}

\address{Department of Mathematics $\&$ Statistics, Texas Tech University, Lubbock, TX 79410-1042, USA}

\address{School of Mathematics and Statistics, University College Dublin, Belfield, Dublin 4, Ireland}

\address{Department of Mathematics and Statistics, University of South Alabama, 411 University Blvd N, Mobile, AL 36688}

\email{dermot.mccarthy@ttu.edu}

\email{robert.osburn@ucd.ie}

\email{straub@southalabama.edu}

\subjclass[2010]{Primary: 11F03; Secondary: 11M32, 14H10}
\keywords{Modular forms, period integrals, supercongruences}

\date{\today}

\begin{abstract}
It is well-known that the Ap{\'e}ry sequences which arise in the irrationality proofs for $\zeta(2)$ and $\zeta(3)$ satisfy many intriguing arithmetic properties and are related to the $p$th Fourier coefficients of modular forms. In this paper, we prove that the connection to modular forms persists for sequences associated to Brown's cellular integrals and state a general conjecture concerning supercongruences.
\end{abstract}

\maketitle

\section{Introduction} \label{sec:intro}

Period integrals on the moduli space $\mathcal{M}_{0,N}$ of curves of genus 0 with $N$ marked points have featured prominently in a variety of mathematical and physical contexts. These period integrals are of particular importance as, for example,  it is now known that they are $\mathbb{Q}$-linear combinations of multiple zeta values  \cite{brown09} thereby proving a conjecture of Goncharov and Manin \cite{gm}, provide an effective computation of a large class of Feynman integrals \cite{bb-feynman}, occur in superstring theory \cite{bss}, \cite{ss} and have potential connections to higher Frobenius limits in relation to the Gamma conjecture \cite{gz}. Recently, Brown \cite{brown-apery} introduced a program where such period integrals play a central role in understanding irrationality proofs of values of the Riemann zeta function. Before discussing this program, we first briefly review the classical situation and some subsequent developments. 

Inspired by Ap{\'e}ry, Beukers \cite{beukers79} gave another proof of the irrationality of $\zeta(2)$ and $\zeta(3)$ by considering the integrals
\begin{equation} \label{zeta2}
  I_n \assign (-1)^n \int_{0}^{1} \int_{0}^{1} \frac{x^n (1-x)^n y^n (1-y)^n}{(1-xy)^{n+1}} \,\, dxdy 
\end{equation}
and
\begin{equation} \label{zeta3}
  J_n \assign \frac12 \int_{0}^{1} \int_{0}^{1} \int_{0}^{1} \frac{x^n (1-x)^n y^n (1-y)^n w^n (1-w)^n}{(1-(1-xy)w)^{n+1}} \,\, dxdydw.
\end{equation}
He showed that
\begin{equation*}
  I_n = a(n) \zeta(2) + \tilde{a}(n) \quad \quad \text{and} \quad \quad J_n = b(n) \zeta(3) + \tilde{b}(n)
\end{equation*}
where $\tilde{a}(n)$ and $\tilde{b}(n)$ are explicit rational numbers and
\begin{equation}\label{AperyNumbers}
  a(n) = \sum_{k=0}^{n} \binom{n}{k}^2 \binom{n+k}{k}, \quad \quad \quad \quad b(n) = \sum_{k=0}^{n} \binom{n}{k}^2 \binom{n+k}{k}^2
\end{equation}
are the {\it Ap{\'e}ry numbers}. Since their introduction, there has been substantial interest in both the intrinsic arithmetic properties of the Ap{\'e}ry numbers and their relationship to modular forms.
For example, consider
\begin{equation*}
  \eta^{6}(4z) =: \sum_{n=1}^{\infty} \alpha(n) q^n,
\end{equation*}
the unique newform in $S_{3}(\Gamma_{0}(16),(\tfrac{-4}{\cdot}))$, where $\eta(z) = q^{1/24} \prod_{n\ge1} (1 - q^n)$ is the Dedekind eta-function, $q=e^{2\pi i z}$ and $z \in \mathbb{H}$.
Ahlgren \cite{ahlgren} showed that, for all primes $p \geq 5$, 
\begin{equation} \label{superapery1}
  a\Bigl(\frac{p-1}{2}\Bigr) \equiv \alpha(p) \pmod{p^2},
\end{equation}
thus confirming a conjecture in \cite{sb}. In \cite{ao}, Ahlgren and Ono proved a conjecture of Beukers \cite{beukers} which stated that, if $p$ is an odd prime, then
\begin{equation} \label{superapery2}
  b\Bigl(\frac{p-1}{2}\Bigr) \equiv \beta(p) \pmod{p^2},
\end{equation}
where 
\begin{equation*}
  \eta^4 (2z) \eta^4 (4z) =: \sum_{n=1}^{\infty} \beta(n) q^n 
\end{equation*}
is the unique newform in $S_{4}(\Gamma_{0}(8))$. Both of these influential results required the existence of an underlying modular Calabi-Yau variety, specifically a $K3$-surface for \eqref{superapery1} and a Calabi-Yau 3-fold for \eqref{superapery2} in order to relate the $p$th Fourier coefficients to finite field hypergeometric series. Concerning arithmetic properties, Coster \cite{coster} proved the supercongruences
\begin{equation} \label{cos1}
  a(mp^r) \equiv a(mp^{r-1}) \pmod{p^{3r}}
\end{equation}
and
\begin{equation} \label{cos2}
  b(mp^r) \equiv b(mp^{r-1}) \pmod{p^{3r}}
\end{equation}
for primes $p \geq 5$ and integers $m$, $r \geq 1$. Other supercongruences of this type have been studied in numerous works (for example, see \cite{at}, \cite{beukers0}, \cite{ccs}, \cite{cooper}, \cite{gessel}, \cite{kc}, \cite{os1}--\cite{oss}, \cite{s}). In order to view \eqref{superapery1}--\eqref{cos2} from a general perspective, we discuss the setup from \cite{brown-apery}.  

Recall that $\mathcal{M}_{0,N}$, $N \geq 4$, is the moduli space of genus zero curves (Riemann spheres) with $N$ ordered marked points $(z_1, \dotsc, z_{N})$. It is the set of $N$-tuples of distinct points $(z_1, \dotsc, z_{N})$ modulo the equivalence relation given by the action of PSL$_{2}$. This action is triply-transitive and so there is a unique representative of each equivalence class such that $z_1=0$, $z_{N-1}=1$ and $z_{N}=\infty$. We introduce simplicial coordinates on $\mathcal{M}_{0,N}$ by setting $t_1=z_2$, $t_2=z_3$, $\dotsc$, $t_{N-3} = z_{N-2}$. This yields the identification
\begin{equation*}
  \mathcal{M}_{0,N} \cong \{ (t_1, \dotsc, t_{N-3}) \in (\mathbb{P}^{1} \setminus \{0, 1, \infty \})^{N-3} \mid t_i \neq t_j \,\, \text{for all} \,\, i \neq j \}.
\end{equation*}
A typical period integral on $\mathcal{M}_{0,N}$ can be given as
\begin{equation} \label{generic}
  \int_{S_N} \prod t_i^{a_i} (1- t_j)^{b_j} (t_i - t_j)^{c_{ij}} dt_1 \dotsc dt_{N-3},
\end{equation}
where $a_i$, $b_j$ and $c_{ij} \in \mathbb{Z}$ are such that \eqref{generic} converges and the simplex
\begin{equation*}
  S_{N} = \{(t_1, \dotsc, t_{N-3}) \in \mathbb{R}^{N-3} : 0 < t_1 < \dotsc < t_{N-3} < 1 \}
\end{equation*}
is a connected component of $\mathcal{M}_{0,N}(\mathbb{R})$. It was proved by Brown \cite{brown09} that the integrals \eqref{generic} are $\mathbb{Q}$-linear combinations of multiple zeta values of weight up to and including $N-3$. That these typically involve multiple zeta values of all weights is an obstruction to irrationality proofs.  For example, generic period integrals on $\mathcal{M}_{0,6}$ yield linear forms in $1$, $\zeta(2)$ and $\zeta(3)$. To ensure the vanishing of  coefficients, we consider a variant of the classical ``dinner table problem" \cite{al}, \cite{poulet}. 

Suppose we have $N \geq 5$ people sitting at a round table. We permute them such that each person has two new neighbors. We represent the new seating arrangement (non-uniquely) by a permutation $\sigma = \sigma_{N}$ on $\{1, 2, \dotsc, N \}$ and write $\sigma=(\sigma(1), \dotsc, \sigma(N))$. A permutation $\sigma$ is called {\it convergent} if no set of $k$ elements in $\{1, \dotsc, N \}$ are simultaneously consecutive for $\{1, \dotsc, N \}$ and $\sigma$ for all $2 \leq k \leq N-2$. If $N \leq 7$, then this is equivalent to the dinner table problem. For $N \geq 8$, this is a new condition. The main idea of \cite{brown-apery} is to associate a rational function $f_{\sigma}$ and a differential ($N-3$)-form $\omega_{\sigma}$ to a given $\sigma$ as follows. Formally, we define
\begin{equation}\label{eq:I:f}
  f_{\sigma} = \prod_{i} \frac{z_i - z_{i+1}}{z_{\sigma(i)} - z_{\sigma(i+1)}} \quad \quad \text{and} \quad \quad \omega_{\sigma} = \frac{dt_1 \dotsc dt_{N-3}}{\prod_{i} z_{\sigma(i)} - z_{\sigma(i+1)}},
\end{equation}
where the product is over all indices $i$ modulo $N$, $z_1=0$ and $z_{N-1}=1$.  We remove factors with $z_{N}$ and then let $z_{i+1} \mapsto t_i$ for $i=1, 2, \dotsc, N-3$.
If we consider the {\it cellular integral}
\begin{equation}\label{eq:I}
  I_{\sigma}(n) = I_{N, \sigma}(n) \assign \int_{S_N} f_{\sigma}^n \,\, \omega_{\sigma}, 
\end{equation}
then $I_{\sigma}(n)$ converges if and only if $\sigma$ is convergent.
For $n=0$, we then obtain the cell-zeta values $\zeta_\sigma(N-3) = I_{\sigma}(0)$ studied in \cite{bcs-cell}, which are multiple zeta values of weight $N-3$. 
More generally, by \cite[Corollary~8.2]{brown-apery}, $I_{\sigma}(n)$ is a $\mathbb{Q}$-linear combination of multiple zeta values of weight less than or equal to $N-3$.
Suppose that this linear combination is of the form $A_{\sigma_N}(n) \zeta_\sigma(N-3)$, with $A_{\sigma_N}(n) \in \mathbb{Q}$, plus a combination of multiple zeta values of weight less than $N-3$ (note that $A_{\sigma_N}(n)$ is necessarily unique if multiple zeta values of different weight are linearly independent over $\mathbb{Q}$; alternatively, since this independence remains open, a unique value can be selected, unconditionally, by working motivically \cite{brown-apery}).
We then say that $A_{\sigma}(n) = A_{\sigma_N}(n)$ is the \emph{leading coefficient} of the cellular integral $I_{\sigma}(n)$.  By construction, $A_{\sigma}(0) = 1$.
Conjecturally, we can define leading coefficients in this fashion for all convergent $\sigma$.
Since we will be concerned with specific permutations $\sigma$, in which case the linear combinations of multiple zeta values can be made explicit, this issue will not disturb our discussion.

For example, if $N=5$, then $\sigma_5=(1,3,5,2,4)$ is the unique (up to dihedral symmetry, see Section~\ref{sec:proof1}) convergent permutation, $I_{\sigma_5}(n)$ recovers Beukers' integral \eqref{zeta2} after a change of variables, and the leading coefficients $A_{\sigma_5}(n)$ are the Ap{\'e}ry numbers $a(n)$. Similarly, for $N=6$, one obtains \eqref{zeta3} upon considering $\sigma_6=(1,5,3,6,2,4)$ and the leading coefficients $A_{\sigma_6}(n)$ are the Ap{\'e}ry numbers $b(n)$. 

This general framework raises some natural questions. Is there an analogue of \eqref{superapery1} and \eqref{superapery2} for $N \geq 7$? Do the leading coefficients $A_{\sigma}(n)$ satisfy supercongruences akin to \eqref{cos1}? The first main result generalizes \eqref{superapery1} to all odd weights greater than or equal to $3$.

\begin{theorem} \label{main1}
  For each odd positive integer $N \geq 5$, there exists a convergent $\sigma_N$ and a modular form $f_{k}(z) =: \sum_{n=1}^{\infty} \gamma_k(n) q^n$ of weight $k=N-2$ such that, for all primes $p \geq 5$,
  \begin{equation}
    A_{\sigma_N} \Bigl(\frac{p-1}{2} \Bigr) \equiv \gamma_{k}(p) \pmod{p^2}.
  \end{equation}
\end{theorem}

As the product of cellular integrals is not necessarily a cellular integral (see Example \ref{eg:non}), powers of the Ap{\'e}ry numbers $a(n)$ are not automatically leading coefficients. We prove Theorem \ref{main1} by carefully constructing an explicit family of convergent $\sigma_N$ such that $A_{\sigma_N}(n) = a(n)^{(N-3)/2}$ and the required modular forms $f_k(z)$. For example, if $N=5$, then $A_{\sigma_5}(n)=a(n)$ and $f_3(z) = \eta^{6}(4z)$. Thus, we recover \eqref{superapery1}. Now, consider the convergent permutation $\sigma_8 = (8, 3, 6, 1, 4, 7, 2, 5)$ and 
$$
\eta^{12}(2z) =: \sum_{n=1}^{\infty} \gamma(n) q^n,
$$
the unique newform in $S_6(\Gamma_0(4))$. Our second main result is a higher weight version of \eqref{superapery2}.

\begin{theorem} \label{main2} 
Let $p$ be an odd prime. Then 
  \begin{equation} \label{even}
    A_{\sigma_8} \Bigl(\frac{p-1}{2} \Bigr) \equiv \gamma(p) \pmod{p^2}.
  \end{equation}
\end{theorem}

In Section 10.1 of \cite{brown-apery}, Brown provided the following table for the number $\mathcal{C}_{N}$ of convergent $\sigma_{N}$ (up to symmetries):

\begin{table}[h]
{\tabulinesep=1mm
\begin{tabu}{|c|ccccccc|} \hline
$N$ & 5 & 6 & 7 & 8 & 9 & 10 & 11 \\ \hline
$\mathcal{C}_{N}$ & 1 & 1 & 5 & 17 & 105 & 771 & 7028 \\ \hline
\end{tabu}} 
\caption{\label{tbl:convergN} $\mathcal{C}_{N}$ for $5 \leq N \leq 11$}
\end{table}

Using the techniques from Section 3, we have obtained explicit binomial sum expressions for all $898$ of the leading coefficients $A_{\sigma_{N}}(n)$ for $7 \leq N \leq 10$. For example, if $N=7$, then the convergent permutations and associated leading coefficients are given in Table \ref{tbl:sigma7LC7}. The ordinary generating functions of these five sequences satisfy fourth order differential equations of Calabi--Yau type, and thus appear as sequences in the table \cite{aesz} (where they are numbered as 193, 243, 101, 198, and 27). The remaining list of convergent permutations and binomial sum expressions for the leading coefficients are available upon request. Based on numerical evidence, we make the following conjecture which suggests that \eqref{cos1} is a generic property of leading coefficients.

\begin{table}[h]
{\tabulinesep=1mm
\begin{tabu}{|c|c|} \hline
$(7,2,4,1,6,3,5)$ & $\displaystyle \sum_{j, k = 0}^n \binom{n}{k}^2 \binom{n}{j}^2 \binom{n + k + j}{n} \binom{k + j}{k}$ \\ \hline
$(7,2,4,6,1,3,5)$ & $ \displaystyle \sum_{j, k = 0}^n \binom{n}{k} \binom{n}{j} \binom{n + k}{n} \binom{n + j}{n} \binom{n + k + j}{n} \binom{n}{k - j}$ \\ \hline
$(7,2,5,1,3,6,4)$ & $\displaystyle \left[ \sum_{k = 0}^n \binom{n}{k}^2 \binom{n + k}{n} \right]^2$ \\ \hline
$(7,2,5,1,4,6,3)$ & $\displaystyle \sum_{j, k = 0}^n \binom{n}{k}^2 \binom{n}{j}^2 \binom{n + k}{n} \binom{k + j}{n}$ \\ \hline
$(7,3,6,2,5,1,4)$ & $\displaystyle \sum_{j, k = 0}^n \binom{n}{k}^2 \binom{n}{j}^2 \binom{n + k}{n} \binom{n - k + j}{n}$ \\ \hline
 \end{tabu}} 
 \caption{\label{tbl:sigma7LC7} Convergent permutations $\sigma_7$ and leading coefficients $A_{\sigma_7}(n)$}
 \end{table}

\begin{conjecture} \label{conj1}
  For each $N \geq 5$ and convergent $\sigma_N$, the leading coefficients $A_{\sigma_N}(n)$ satisfy
  \begin{equation}
    A_{\sigma_N}(mp^r) \equiv A_{\sigma_N}(mp^{r-1}) \pmod{p^{3r}}
  \end{equation}
  for all primes $p \geq 5$ and integers $m$, $r \geq 1$.
\end{conjecture}

Given Theorems \ref{main1} and \ref{main2}, it is natural to wonder if such supercongruences hold for all leading coefficients $A_{\sigma_N}(n)$. Do they arise from $L$-series attached to Galois representations? It appears that a result similar to \eqref{superapery2} and \eqref{even} also holds for a convergent $\sigma_{10}$ and modular form of weight 8. This and other observations will be the subject of forthcoming work. Finally, it would be of interest to examine Theorems \ref{main1} and \ref{main2} and Conjecture \ref{conj1} from the recent ``motivic" perspective of \cite{rrvw}. 

The paper is organized as follows. In Section~\ref{sec:proof1}, we first discuss the necessary background on dihedral symmetries, the multiplicative structure of cellular integrals and modular forms with complex multiplication, then prove Theorem \ref{main1}. In Section~\ref{sec:binomial}, we explain how to derive multiple binomial sum representations for the leading coefficients. In Section~\ref{sec:proof2}, we prove some preliminary results on combinatorial congruences, recall finite field hypergeometric series and then prove Theorem \ref{main2}. 

\section{Proof of Theorem \ref{main1}} \label{sec:proof1}

\subsection{Dihedral symmetries and multiplicative structures} \label{sec:dihedral}

Let $\Sigma_N$ be the symmetric group on $\{ 1, 2, \ldots, N \}$. A
{\emph{dihedral structure}} on $\{ 1, 2, \ldots, N \}$ is an equivalence class
of permutations $\sigma \in \Sigma_N$, where the equivalences are generated by
\begin{eqnarray*}
  (\sigma_1, \sigma_2, \ldots, \sigma_N) & \sim & (\sigma_2, \sigma_3, \ldots,
  \sigma_N, \sigma_1)\\
  (\sigma_1, \sigma_2, \ldots, \sigma_N) & \sim & (\sigma_N, \sigma_{N - 1},
  \ldots, \sigma_1) .
\end{eqnarray*}
A {\emph{configuration}} on $\{ 1, 2, \ldots, N \}$ is an equivalence class
$[\delta, \delta']$ of pairs $(\delta, \delta')$ of dihedral structures modulo
the equivalence relations
\begin{equation*}
  (\delta, \delta') \sim (\sigma \delta, \sigma \delta')
\end{equation*}
for $\sigma \in \Sigma_N$. Associated to every permutation $\sigma \in
\Sigma_N$ is the configuration $[\sigma] \assign [\operatorname{id}, \sigma]$.
Clearly, every configuration can be thus represented. Indeed, the
configurations on $\{ 1, 2, \ldots, N \}$ can be identified with the double
cosets $D_{2 N} \backslash \Sigma_N / D_{2 N}$, where $D_{2 N}$ are the
dihedral permutations of $\Sigma_N$. Henceforth, we will not make a distinction between permutations and configurations. The {\emph{dual}} of a configuration
$[\delta, \delta']$ is the configuration $[\delta', \delta]$. Equivalently,
the dual of $[\sigma]$ is $[\sigma]^{\vee} \assign [\sigma^{- 1}]$. For
further details, we refer to \cite[Section~3.1]{brown-apery}. The notion of
a configuration is important for our considerations because, up to a possible
factor of $(- 1)^n$, the cellular integral $I_{\sigma} (n)$ only depends on
the configuration $[\sigma]$.

In \cite[Section~6]{brown-apery}, Brown describes the following partial
multiplication on pairs of dihedral structures. A pair of dihedral structures
$(\delta, \delta')$ on $\{ 1, 2, \ldots, N \}$ is {\emph{multipliable}} along
the triple $t = (t_1, t_2, t_3)$, where $t_1, t_2, t_3 \in \{ 1, 2, \ldots, N
\}$ are distinct, if the elements $t_1, t_2, t_3$ are consecutive in $\delta$
and the elements $t_1, t_3$ are consecutive in $\delta'$.

Let $(\alpha, \alpha')$ and $(\beta, \beta')$ be dihedral structures on $X= \{ 1,
2, \ldots, N \}$ and $Y = \{ 1, 2, \ldots, M \}$, respectively. If $(\alpha,
\alpha')$ is multipliable along $s$, and $(\beta', \beta)$ (the dual of
$(\beta, \beta')$) is multipliable along $t$, then the product
\begin{equation*}
  (\gamma, \gamma') : = (\alpha, \alpha') \star_{s, t} (\beta, \beta')
\end{equation*}
is a pair of dihedral structures on the disjoint union $Z$ of $X$ and $Y$,
where each of the three elements $s_1, s_2, s_3$ from $X$ gets identified with the corresponding
element $t_1, t_2, t_3$ from $Y$. The dihedral structure $\gamma$ is the
unique structure such that its restriction to $X$ (respectively, $Y$)
coincides with $\alpha$ (respectively, $\beta$). We say that $\gamma$ was
obtained by shuffling $\alpha$ and $\beta$ (more general such shuffles are
described in \cite[Section~2.3.1]{bcs-cell}).  $\gamma'$ is likewise obtained
by shuffling $\alpha'$ and $\beta'$.

In order to work explicitly, we identify this disjoint union of $X$ and $Y$
with the set $Z = \{ 1, 2, \ldots, M + N - 3 \}$ in such a way that $Y$ is the
natural subset of $Z$, while the $N - 3$ elements $1, 2, \ldots, N$ of $X$,
with $s_1, s_2, s_3$ removed, are identified, in that order, with the elements
$M + 1, M + 2, \ldots, M + N - 3$ of $Z$. We note that different choices for
this identification lead to the same configuration $[\gamma, \gamma']$.

\begin{example}
  \label{eg:mult:s5:ab}For illustration, let us multiply $(\alpha, \alpha') =
  ((1, 2, 3, 4, 5), (1, 3, 5, 2, 4))$ with $(\beta, \beta') = (\alpha,
  \alpha')$ along $s = (1, 2, 3)$, $t = (4, 2, 5)$. Here, $X = \{ 1, 2, 3, 4,
  5 \}$, $Y = \{ 1, 2, 3, 4, 5 \}$ and $Z = \{ 1, 2, \ldots, 7 \}$. The
  elements $1, 2, 3, 4, 5$ of $Y$ are identified with the elements $1, 2, 3,
  4, 5$ of $Z$, and the elements $1, 2, 3, 4, 5$ of $X$ are identified with
  the elements $4, 2, 5, 6, 7$ of $Z$. In other words, we replace $(\alpha,
  \alpha')$ with $(\alpha, \alpha') = ((4, 2, 5, 6, 7), (4, 5, 7, 2, 6))$. In
  order to obtain $\gamma$, we shuffle $\alpha = (4, 2, 5, 6, 7)$ and $\beta =
  (1, 2, 3, 4, 5)$ such that both dihedral structures are preserved (because
  of the conditions of multipliability there is a unique such shuffle). The
  result is $\gamma = (1, 2, 3, 4, 7, 6, 5)$. Likewise, shuffling $\alpha' =
  (4, 5, 7, 2, 6)$ and $\beta' = (1, 3, 5, 2, 4)$ results in $\gamma' = (1, 3,
  5, 7, 2, 6, 4)$. We refer to \cite[Example~6.3]{brown-apery} for another
  example accompanied by a helpful illustration.
\end{example}

The main interest in this partial multiplication stems from the following
result. For its statement, recall that every configuration $[\gamma, \gamma']$
can be represented as $[\sigma] = [\operatorname{id}, \sigma]$. Since, up to a sign,
the cellular integral $I_{\sigma}$ only depends on the configuration
$[\sigma]$, we may write $I_{[\gamma, \gamma']} = I_{\sigma}$. Likewise, we
write $A_{[\gamma, \gamma']} = A_{\sigma}$ for the corresponding leading
coefficients.

\begin{proposition}[\cite{brown-apery}, Proposition~6.5]
  \label{prop:mult:I}Suppose that
  \begin{equation*}
    (\gamma, \gamma') = (\alpha, \alpha') \star_{s, t} (\beta, \beta')
  \end{equation*}
  for some choice of $s, t$. Then, for all $n \geq 0$, possibly up to a
  sign,
  \begin{equation*}
    I_{[\gamma, \gamma']} (n) = I_{[\alpha, \alpha']} (n) I_{[\beta, \beta']}
     (n) .
  \end{equation*}
  In particular, $A_{[\gamma, \gamma']} (n) = A_{[\alpha, \alpha']} (n)
  A_{[\beta, \beta']} (n)$.
\end{proposition}

\begin{example}
  \label{eg:mult:s5:s7}Recall that, for $\sigma_5 = (1, 3, 5, 2, 4)$, the
  leading terms $A_{\sigma_5} (n)$ of the corresponding cellular integrals are
  the Ap\'ery numbers $a (n)$. The configuration $[\sigma_5]$ is represented
  by the pairs of dihedral structures $(\alpha, \alpha') = (\beta, \beta')$
  from Example~\ref{eg:mult:s5:ab}. Let $(\gamma, \gamma')$ be the product of
  these two pairs along $s, t$ as chosen in Example~\ref{eg:mult:s5:ab}.
  Observe that, as a configuration, $[\gamma, \gamma'] = [\sigma_7]$ with
  $\sigma_7 = (1, 3, 7, 5, 2, 6, 4)$. By Proposition~\ref{prop:mult:I},
  $A_{\sigma_7} (n) = a (n)^2$.
\end{example}

Similarly, many leading coefficients $A_{\sigma} (n)$ are products of leading
coefficients of lower order. However, it is not the case that products of
leading coefficients are always leading coefficients.

\begin{example} \label{eg:non}
  Consider the configuration $\sigma_8 = (8, 3, 6, 1, 4, 7, 2, 5)$ featured in
  Theorem~\ref{main2} and Section~\ref{sec:binomial}. This configuration is
  self-dual and not multipliable along any choice of triple $t$. Indeed, we
  confirm numerically that the product $a (n) A_{\sigma_8} (n)$ is not a
  leading coefficient $A_{\sigma_{10}} (n)$ for any configuration
  $\sigma_{10}$.
\end{example}

Nevertheless, the next result guarantees that all positive integer powers
of the Ap\'ery numbers $a(n)$ are leading coefficients.

\begin{proposition}
  \label{prop:mult:a}If $\rho = (N - 1, \rho_0, N, \rho_1, \rho_2, \ldots,
  \rho_{N - 3}) \in \Sigma_N$, then
  \begin{equation*}
    A_{\tau} (n) = a (n) A_{\rho} (n),
  \end{equation*}
  where $\tau = (N + 1, \rho_0 + 1, N + 2, N, \rho_{N - 3} + 1, \rho_{N - 4} +
  1, \ldots, \rho_1 + 1, 1)$.
\end{proposition}

Note that $\tau$ is of the same shape of $\rho$, so that the result can be
iterated to conclude that $a (n)^{\lambda} A_{\rho} (n)$ is a leading
coefficient for any integer $\lambda \geq 0$.

\begin{proof}
  Let $(\alpha, \alpha') = ((1, 2, 3, 4, 5), (1, 3, 5, 2, 4))$ and $(\beta,
  \beta') = ((1, 2, \ldots, N), \rho)$. Note that $(\alpha, \alpha')$ is
  multipliable along $s = (1, 2, 3)$, and that the dual of $(\beta, \beta')$
  is multipliable along $t = (N - 1, \rho_0, N)$. Let $(\gamma, \gamma') =
  (\alpha, \alpha') \star_{s, t} (\beta, \beta')$.
  
  To compute this product, we proceed as in Example~\ref{eg:mult:s5:ab} and
  replace $(\alpha, \alpha')$ with $(\alpha, \alpha') = ((N - 1, \rho_0, N, N
  + 1, N + 2), (N - 1, N, N + 2, \rho_0, N + 1))$. Then, shuffling $\alpha =
  (N - 1, \rho_0, N, N + 1, N + 2)$ and $\beta = (1, 2, \ldots, N)$ such that
  both dihedral structures are preserved, we obtain
  \begin{equation*}
    \gamma = (1, 2, \ldots, N - 1, N + 2, N + 1, N) .
  \end{equation*}
  Likewise, shuffling $\alpha' = (N - 1, N, N + 2, \rho_0, N + 1)$ and $\beta'
  = (N - 1, \rho_0, N, \rho_1, \rho_2, \ldots, \rho_{N - 3})$ results in
  \begin{equation*}
    \gamma' = (N - 1, N + 1, \rho_0, N + 2, N, \rho_1, \rho_2, \ldots,
     \rho_{N - 3}) .
  \end{equation*}
  Since $[\alpha, \alpha'] = [\sigma_5]$ and $[\beta, \beta'] = [\rho]$,
  Proposition~\ref{prop:mult:I} implies that
  \begin{equation*}
    A_{[\gamma, \gamma']} (n) = a (n) A_{\rho} (n),
  \end{equation*}
  so that it only remains to observe that $[\gamma, \gamma'] = [\tau]$. First,
  swapping $N$ and $N + 2$, we find that
  \begin{equation*}
    [\gamma, \gamma'] = [(N - 1, N + 1, \rho_0, N, N + 2, \rho_1, \rho_2,
     \ldots, \rho_{N - 3})] .
  \end{equation*}
  Finally, as configurations, we have
  \begin{eqnarray*}
    &  & [(N - 1, N + 1, \rho_0, N, N + 2, \rho_1, \rho_2, \ldots, \rho_{N -
    3})]\\
    & = & [(N, N + 2, \rho_0 + 1, N + 1, 1, \rho_1 + 1, \rho_2 + 1, \ldots,
    \rho_{N - 3} + 1)]\\
    & = & [(N + 1, \rho_0 + 1, N + 2, N, \rho_{N - 3} + 1, \rho_{N - 4} + 1,
    \ldots, \rho_1 + 1, 1)],
  \end{eqnarray*}
  which is $[\tau]$, as claimed.
\end{proof}

\begin{example} \label{powers}
  Let us illustrate with $\rho = (4, 2, 5, 3, 1)$. Since $A_{\rho} (n) = a
  (n)$, Proposition~\ref{prop:mult:a} implies that
  \begin{equation*}
    A_{\sigma_7} (n) = a (n)^2, \quad \sigma_7 = (6, 3, 7, 5, 2, 4, 1) .
  \end{equation*}
  Note that this configuration agrees with the one obtained in
  Example~\ref{eg:mult:s5:s7}. Iterating Proposition~\ref{prop:mult:a}, we
  find
  \begin{eqnarray*}
    A_{\sigma_9} (n) = a (n)^3, & \quad & \sigma_9 = (8, 4, 9, 7, 2, 5, 3, 6,
    1),\\
    A_{\sigma_{11}} (n) = a (n)^4, &  & \sigma_{11} = (10, 5, 11, 9, 2, 7, 4,
    6, 3, 8, 1) .
  \end{eqnarray*}
\end{example}

In fact, the family of configurations in Example \ref{powers} can be made
explicit as follows. For a positive integer $M$ and configuration $(a_1, a_2, \ldots, a_{N})$, we write
$M + (a_1, a_2, \ldots, a_N)$ for $(M + a_1, M + a_2, \ldots, M + a_N)$. 

\begin{corollary}
  \label{cor:mult:a}Let $M \geq 2$ be an integer. Then
  \begin{equation*}
    A_{\sigma_{2 M + 1}} (n) = a (n)^{M - 1},
  \end{equation*}
  where the configuration $\sigma_{2 M + 1}$ is
  \begin{equation*}
    \sigma_{2 M + 1} = M + (M, 0, M + 1, M - 1, - (M - 2), M - 3, \ldots, \pm
     1, \mp 1, \pm 2, \ldots, - (M - 1)).
  \end{equation*}
  \end{corollary}

\begin{proof}
  The statement is clearly true for $M = 2$, in which case $\sigma_5 = (4, 2,
  5, 3, 1)$. The claim then follows from Proposition~\ref{prop:mult:a} by
  induction.
\end{proof}

\subsection{Modular forms with complex multiplication and Hecke characters}\label{sec_Prelim_CM}
In this section, we recall some properties of modular forms with complex multiplication and Hecke characters. For more details, see \cite{Ri}.

Suppose $\psi$ is a nontrivial real Dirichlet character with corresponding quadratic field $K$. A newform $f(z)= \sum_{n=1}^{\infty} \gamma(n) \, q^n$, where $q \assign e^{2 \pi i z}$, has complex multiplication (CM) by $\psi$, or by $K$, if $\gamma(p) = \psi(p) \, \gamma(p)$ for all primes $p$ in a set of density one. 

By the work of Hecke and Shimura we can construct CM newforms using Hecke characters. Let $K=\mathbb{Q}(\sqrt{-d})$ be an imaginary quadratic field with discriminant $D$, and let $\mathcal{O}_K$ be its ring of integers. For an ideal $\mathfrak{f} \in \mathcal{O}_K$, let $I(\mathfrak{f})$ denote the group of fractional ideals prime to $\mathfrak{f}$. A Hecke character of weight $k$ and modulo $\mathfrak{f}$ is a homomorphism
$\displaystyle
\Phi : I(\mathfrak{f}) \to \mathbb{C}^{*}$,
satisfying
$\displaystyle
\Phi(\alpha \, \mathcal{O}_K)= \alpha^{k-1} 
$
when
$
\alpha \equiv^{\times} 1 \pmod {\mathfrak{f}}.
$
Let $N(\mathfrak{a})$ denote the norm of the ideal $\mathfrak{a}$. Then,
\begin{align*}
  f(z) \assign \sum_{\mathfrak{a}} \Phi(\mathfrak{a}) q^{N(\mathfrak{a})} = \sum_{n=1}^{\infty} \gamma(n) q^n,
\end{align*}
where the sum is over all ideals $\mathfrak{a}$ in $\mathcal{O}_K$ prime to $\mathfrak{f}$, is a Hecke eigenform of weight $k$ on ${\Gamma_0(|D| \cdot N(\mathfrak{f}))}$ with Nebentypus 
$\chi(n)=\left(\frac{D}{n}\right) \frac{\Phi(n \, \mathcal{O}_K)}{n^{k-1}}$. Here, $\left( \tfrac{a}{n}\right)$ is the Kronecker symbol.
Furthermore, $f$ has CM by $K$. We call $\mathfrak{f}$ the conductor of $\Phi$ if $\mathfrak{f}$ is minimal, i.e., if $\Phi$ is defined modulo $\mathfrak{f}^{\prime}$ then $\mathfrak{f} \mid \mathfrak{f}^{\prime}$. If $\mathfrak{f}$ is the conductor of $\Phi$ then $f(z)$ is a newform. From \cite{Ri}, we also know that every CM newform comes from a Hecke character in this way. 

We will see that Theorem \ref{main1} is a consequence of Corollary \ref{cor:mult:a} and Corollary \ref{cor_AperyModCoeffs}. 

\begin{theorem}\label{thm_CM}
Let $k \geq 2$ be a positive integer. Then there exists a weight $k$ CM newform
$$
f_k(z) =: \sum_{n=1}^{\infty} \gamma_k(n) q^n
\in
\begin{cases}
S_k(\Gamma_0(32)), & \text{if $k$ is even},\\
S_k(\Gamma_0(4),(\frac{-4}{\cdot})), & \text{if $k \equiv 1 \imod 4$},\\
S_k(\Gamma_0(16),(\frac{-4}{\cdot})), & \text{if $k \equiv 3 \imod 4$},
\end{cases}
$$
such that, for any odd prime $p$,
$$
\gamma_k(p)
=
\begin{cases}
(-1)^{\frac{(x+y-1)(k-1)}{2}} \left[(x+iy)^{k-1} + (x-iy)^{k-1} \right], & \text{if $p \equiv 1 \imod 4$, $p= x^2+y^2$, $x$ odd},\\
0, & \text{if $p \equiv 3 \imod 4$}.
\end{cases}
$$
\end{theorem}

\begin{proof}
For $k$ in each equivalence class modulo $4$, we will define a Hecke character $\Psi_k$ and construct the required CM newform $f_k$, using the methodology outlined above.

For an ideal $\mathfrak{f} \in \mathcal{O}_K$, let $I(\mathfrak{f})$ denote the group of fractional ideals prime to $\mathfrak{f}$, and let $J(\mathfrak{f})$ be the subset of principal fractional ideals whose generator is multiplicatively congruent to $1$ modulo $\mathfrak{f}$, i.e., $J(\mathfrak{f})=\{ (\alpha) \in I(\mathfrak{f}) \mid \alpha \equiv^{\times} 1 \pmod {\mathfrak{f}} \}$.

Let $K=\mathbb{Q}(\sqrt{-1})$ which has discriminant $D=-4$ and whose ring of integers is $\mathcal{O}_K = \mathbb{Z}[\sqrt{-1}]$, which is a principal ideal domain. Therefore, all fractional ideals of $K$ are also principal, and are of the form $\frac{1}{m} (\alpha)$ where $m \in \mathbb{Z} \setminus \{0\}$ and $\alpha \in \mathcal{O}_K$.

Case 1: $k \equiv 1 \pmod 4$. Let $\mathfrak{f}= (1)$. Then $I(\mathfrak{f})=J(\mathfrak{f})$ is the set of all fractional ideals. We define the Hecke character $\Phi_k : I((1)) \to \mathbb{C}^{*}$ of weight $k$ and conductor $(1)$ by
$$
\Phi_k \left( \tfrac{1}{m} (\alpha) \right)
=
\left( \tfrac{\alpha}{m}  \right)^{k-1}.
$$
Therefore, 
$$
f_k(z) \assign \sum_{\mathfrak{a}} \Phi_k(\mathfrak{a}) q^{N(\mathfrak{a})} = \sum_{n=1}^{\infty} \gamma_k(n) q^n \in S_k(\Gamma_0(4),(\tfrac{-4}{\cdot}))
$$
is a CM newform and, for $p$ an odd prime, 
$$
\gamma_k(p) = 
\begin{cases}
(x+iy)^{k-1} +  (x-iy)^{k-1}, & \text{if $p \equiv 1 \imod 4$, $p= x^2+y^2$, $x$ odd},\\
0, & \text{if $p \equiv 3 \imod 4$}.
\end{cases}
$$

\noindent
Case 2: $k \equiv 3 \pmod 4$. Let $\mathfrak{f}= (2)$. Then
$$
I((2)) = \{ \tfrac{1}{m} (\alpha) \mid m \in \mathbb{Z} \text{ odd}, \alpha=x+iy \in \mathbb{Z}[\sqrt{-1}], \text{$x$ and $y$ different parity} \}
$$
and
$$
J((2)) = \{ \tfrac{1}{m} (\alpha) \mid m \in \mathbb{Z} \text{ odd}, \alpha=x+iy \in \mathbb{Z}[\sqrt{-1}], \text{$x$ odd, $y$ even} \}.
$$

We define the Hecke character $\Phi_k : I((2)) \to \mathbb{C}^{*}$ of weight $k$ and conductor $(2)$ by
$$
\Phi_k \left( \tfrac{1}{m} (x+iy) \right)
=
\left( \tfrac{x+iy}{m}  \right)^{k-1} \cdot (-1)^y. 
$$
Therefore, 
$$
f_k(z) \assign \sum_{\mathfrak{a}} \Phi_k(\mathfrak{a}) q^{N(\mathfrak{a})} = \sum_{n=1}^{\infty} \gamma_k(n) q^n \in S_k(\Gamma_0(16),(\tfrac{-4}{\cdot}))
$$
is a CM newform, and for $p$ an odd prime, 
$$
\gamma_k(p) = 
\begin{cases}
(x+iy)^{k-1} +  (x-iy)^{k-1},  & \text{if $p \equiv 1 \imod 4$, $p= x^2+y^2$, $x$ odd},\\
0, & \text{if $p \equiv 3 \imod 4$}.
\end{cases}
$$

\noindent
Case 3: $k$ even. Let $\mathfrak{f}= (2+2i)$. Then
$$
I((2+2i)) = \{ \tfrac{1}{m} (\alpha) \mid \text{$m \in \mathbb{Z}$ odd, $\alpha=x+iy \in \mathbb{Z}[\sqrt{-1}]$, $x$ and $y$ different parity} \}
$$
and
\begin{multline*}
  J((2+2i)) = \{ \tfrac{1}{m} (\alpha) \mid \text{$m \in \mathbb{Z}$ odd, $\alpha=x+iy \in \mathbb{Z}[\sqrt{-1}]$, $x$ odd, $y$ even, such that}\\ 
  \text{ if $m \equiv x \imod{4}$ then $y \equiv 0 \imod{4}$, otherwise $y \equiv 2 \imod{4}$} \}.
\end{multline*}
We define the Hecke character $\Phi_k : I((2+2i)) \to \mathbb{C}^{*}$ of weight $k$ and conductor $(2+2i)$ by
$$
\Phi_k \left( \tfrac{1}{m} (x+iy) \right)
=
\left( \tfrac{x+iy}{m}  \right)^{k-1} \chi(x+iy)^{k-1}  \cdot (-1)^y \cdot (-1)^{\frac{m-1}{2}},
$$
where
$$
\chi(x+iy)=
(-1)^{\frac{x+y-1}{2}} \cdot \begin{cases}
1, & \text{if $x$ odd, $y$ even},\\
i, & \text{if $x$ even, $y$ odd}.
\end{cases} 
$$
Therefore, 
$$
f_k(z) \assign \sum_{\mathfrak{a}} \Phi_k(\mathfrak{a}) q^{N(\mathfrak{a})} = \sum_{n=1}^{\infty} \gamma_k(n) q^n \in S_k(\Gamma_0(32))
$$
is a CM newform and, for $p$ an odd prime, 
$$
\gamma_k(p) = 
\begin{cases}
(-1)^{\frac{x+y-1}{2}} \left[ (x+iy)^{k-1} +  (x-iy)^{k-1}  \right], & \text{if $p \equiv 1 \imod 4$, $p= x^2+y^2$, $x$ odd},\\
0, & \text{if $p \equiv 3 \imod 4$}.
\end{cases}
$$
\end{proof}

\begin{remark} When $k$ is odd, an alternative expression for the CM newform $f_{k}(z)$ constructed in Theorem \ref{thm_CM}
is given by the binary theta series
\begin{equation}
f_k (z) = \frac{1}{4}  \sum_{(n, m) \in \mathbb{Z}^2} (- 1)^{m (k - 1) /
  2} (n - i m)^{k - 1} q^{n^2 + m^2}. \label{eq:fk}
\end{equation}
The formulation \eqref{eq:fk} was recently used to relate the $L$-series of $f_k$ at $k-1$ to an interpolated version of the leading coefficients in Corollary \ref{cor:mult:a}. For further details, see \cite{osL}.
\end{remark}

\begin{corollary}\label{cor_Coeffs}
Let $k \geq 2$ be a positive integer and let
$f_k(z) =: \sum_{n=1}^{\infty} \gamma_k(n) q^n$
be the weight $k$ CM newform described in Theorem~\ref{thm_CM}. Then, for any odd prime $p$ and integer $m \geq 1$,
\begin{align*}
\gamma_k(p)^m &= \sum_{t=0}^{\lfloor \frac{m-1}{2} \rfloor} \binom{m}{t} p^{t(k-1)} \gamma_{(m-2t)(k-1)+1}(p)
+ 
\begin{cases}
\displaystyle \binom{m}{\frac{m}{2}} p^{\frac{m}{2}(k-1)}, & \text{if $p \equiv 1 \imod 4$ and $m$ even},\\
0, & \text{otherwise}.
\end{cases}
\end{align*}
\end{corollary}

\begin{proof}
This follows from a simple application of the binomial theorem on the expression for $\gamma_k(p)$ given in Theorem \ref{thm_CM}, while noting that 
$(-1)^{\frac{x+y-1}{2}} (x+iy) \cdot (-1)^{\frac{x+y-1}{2}} (x-iy) = p.$
\end{proof}

\begin{corollary}\label{cor_AperyModCoeffs}
Let $l \geq 1$ be a positive integer and define $k:=2l+1$. Let
$f_k(z) =: \sum_{n=1}^{\infty} \gamma_k(n) q^n$
be the weight $k$ CM newform described in Theorem~\ref{thm_CM}.
Let $a(n) = \sum_{j=0}^{n} \binom{n}{j}^2 \binom{n+j}{j}$ be the {\it Ap{\'e}ry numbers} introduced in~(\ref{AperyNumbers}).
Then, for primes $p \geq 5$,
$$
a(\tfrac{p-1}{2})^l \equiv \gamma_{k}(p) \pmod{p^2}.
$$   
\end{corollary}

\begin{proof}
By (\ref{superapery1}) and Corollary \ref{cor_Coeffs},
\begin{equation*}
a(\tfrac{p-1}{2})^l 
\equiv \gamma_3(p)^l
\equiv \gamma_{k}(p) \pmod{p^2}.
\end{equation*}
\end{proof}

We can now prove Theorem \ref{main1}.

\begin{proof}[Proof of Theorem \ref{main1}]
Let $l \geq 1$ be a positive integer and define $k=2l+1$ and $N=2l+3$. Consider $f_{k}(z) =: \sum_{n=1}^{\infty} \gamma_{k}(n) q^n$ as in Theorem \ref{thm_CM}. By Corollary \ref{cor:mult:a}, there exists a convergent $\sigma_{N}$ whose leading coefficient $A_{\sigma_N}(n)$ satisfies
\begin{equation} \label{atol}
A_{\sigma_N}(n) = a(n)^{l}.
\end{equation}
Thus, the result then follows from (\ref{atol}) and Corollary \ref{cor_AperyModCoeffs}.
\end{proof}

\section{Multiple binomial sums}\label{sec:binomial}

\subsection{Proving multiple binomial sum representations for $A_{\sigma} (n)$}

In this section, we outline how to obtain multiple binomial sum representations
for the leading coefficients $A_{\sigma} (n)$. Our discussion applies to any
configuration, but we proceed with the configuration $\sigma_8 = (8, 3, 6, 1,
4, 7, 2, 5)$ for $N = 8$ which is relevant to Theorem~\ref{main2}. This is the
self-dual configuration $_8 \pi_6$ in Brown's notation
\cite[Section~10.1.4]{brown-apery}. The leading coefficients $A_{\sigma_8}
(n)$ have initial terms
\begin{equation*}
  1, 33, 8929, 4124193, 2435948001, 1657775448033, \ldots
\end{equation*}
and the following particularly symmetric binomial sum representation.

\begin{proposition}
  \label{prop:A86}For the configuration $\sigma_8 = (8, 3, 6, 1, 4, 7, 2, 5)$,
  the leading coefficients are
  \begin{equation}
    A_{\sigma_8} (n) = \sum_{\substack{
      k_1, k_2, k_3, k_4 = 0 \\
      k_1 + k_2 = k_3 + k_4
    }}^{n} \prod_{i = 1}^4 \binom{n}{k_i} \binom{n + k_i}{k_i} .
    \label{eq:A86}
  \end{equation}
\end{proposition}

\begin{proof}
  For $\sigma=\sigma_8$, consider the cellular integral $I_{\sigma} (n) = \int_{S_8} f_{\sigma}^n \,\,
  \omega_{\sigma}$ where
  \begin{equation}
    f_{\sigma} = \frac{(z_1 - z_2) (z_2 - z_3) (z_3 - z_4) (z_4 - z_5) (z_5 -
    z_6) (z_6 - z_7) (z_7 - z_8) (z_8 - z_1)}{(z_8 - z_3) (z_3 - z_6) (z_6 -
    z_1) (z_1 - z_4) (z_4 - z_7) (z_7 - z_2) (z_2 - z_5) (z_5 - z_8)} .
    \label{eq:f:z:86}
  \end{equation}
  As in Section \ref{sec:intro}, we let $z_1 = 0$, $z_7 =
  1$, $z_8 = \infty$. Then, in the coordinates $t_1 = z_2, t_2 = z_3, t_3 =
  z_4, t_4 = z_5, t_5 = z_6$, we have
  \begin{equation*}
    f_{\sigma} = \frac{(- t_1) (t_1 - t_2) (t_2 - t_3) (t_3 - t_4) (t_4 -
     t_5) (t_5 - 1)}{(t_2 - t_5) (t_5) (- t_3) (t_3 - 1) (1 - t_1) (t_1 -
     t_4)},
  \end{equation*}
  and
  \begin{equation*}
    \omega_{\sigma} = \frac{d t_1 d t_2 d t_3 d t_4
     d t_5}{(t_2 - t_5) (t_5) (- t_3) (t_3 - 1) (1 - t_1) (t_1 - t_4)} .
  \end{equation*}
  The domain $S_8$ of integration then consists of all $(t_1, t_2, \ldots,
  t_5) \in \mathbb{R}^5$ such that $0 < t_1 < t_2 < \ldots < t_5 < 1$.
  Algorithmic approaches to computing explicit period integrals, such as
  $I_{\sigma} (n)$ for specific values of $n$, are described in
  \cite{bb-feynman} or \cite{panzer-hyperint}. In particular, Panzer
  implemented his symbolic integration approach \cite{panzer-hyperint} using
  hyperlogarithms in a Maple package called \texttt{HyperInt}. Using this
  package, we explicitly evaluate $I_{\sigma} (n)$ in terms of multiple zeta
  values for several small values of $n$, and obtain:
  \begin{eqnarray*}
    I_{\sigma} (0) & = & 16 \zeta (5) - 8 \zeta (3) \zeta (2)\\
    I_{\sigma} (1) & = & 33 I_{\sigma} (0) - 432 \zeta (3) + 316 \zeta (2) -
    26\\
    I_{\sigma} (2) & = & 8929 I_{\sigma} (0) - 117500 \zeta (3) +
    \tfrac{515189}{6} \zeta (2) - \tfrac{331063}{48}\\
    I_{\sigma} (3) & = & 4124193 I_{\sigma} (0) - 54272204 \zeta (3) +
    \tfrac{10708231609}{270} \zeta (2) - \tfrac{12385477271}{3888}\\
    I_{\sigma} (4) & = & 2435948001 I_{\sigma} (0) - 32055790815 \zeta (3) +
    \tfrac{23612586361625}{1008} \zeta (2) - \tfrac{78031593554765}{41472}.
  \end{eqnarray*}
  As proven by Brown in \cite[Section~4]{brown-apery},
  the integrals $I_{\sigma} (n)$ satisfy a linear recurrence with polynomial
  coefficients. Slightly more specifically, the ordinary generating function
  \begin{equation*}
    F_{\sigma} (x) \assign \sum_{n = 0}^{\infty} I_{\sigma} (n) x^n
  \end{equation*}
  satisfies a Picard--Fuchs differential equation. Again, in each specific
  instance, this differential equation can be obtained algorithmically. A
  particularly efficient such approach is an extension of the Griffiths--Dwork
  reduction method due to Lairez \cite{lairez-periods} for computing periods
  of rational functions. To apply this method, we observe that
  \begin{equation*}
    F_{\sigma} (x) = \int_{S_8}
     \frac{\omega_{\sigma}}{1 - f_{\sigma} x} = \int_{S_8} \frac{d t_1
     d t_2 d t_3 d t_4 d t_5}{A - B x},
  \end{equation*}
  where
  \begin{eqnarray*}
    A & = & (t_2 - t_5) (t_5) (- t_3) (t_3 - 1) (1 - t_1) (t_1 - t_4),\\
    B & = & (- t_1) (t_1 - t_2) (t_2 - t_3) (t_3 - t_4) (t_4 - t_5) (t_5 - 1).
  \end{eqnarray*}
  Lairez's method (implemented in Magma) then successfully determines
  a Fuchsian differential equation of order $7$ satisfied by $F_{\sigma} (x)$. 
  This differential equation has a two-dimensional space of analytic solutions around $x = 0$. As a consequence, this differential equation together with the two values $I_{\sigma} (0)$, $I_{\sigma} (1)$, explicitly obtained above, determines the values of the cellular integrals $I_{\sigma} (n)$ for all $n \geq 2$.
  
  Alternatively, the differential equation for the generating function
  $F_{\sigma} (x)$ translates directly into a recurrence of order $12$ for the
  coefficients $I_{\sigma} (n)$ and, hence, for the leading coefficients
  $A_{\sigma} (n)$. That is, we find that
  \begin{equation*}
    1521 n (n + 1)^5 (2 n + 1) A_{\sigma} (n + 1) = \sum_{j = 0}^{11} q_j (n)
     A_{\sigma} (n - j)
  \end{equation*}
  for certain polynomials $q_j (x) \in \mathbb{Z} [x]$ of degree $7$. To
  complete the proof, it therefore only remains to verify that the numbers
  \begin{equation*}
    B_{\sigma} (n) \assign \sum_{\substack{
       k_1, k_2, k_3, k_4 =0 \\
       k_1 + k_2 = k_3 + k_4
     }}^{n} \prod_{i = 1}^4 \binom{n}{k_i} \binom{n + k_i}{n}
  \end{equation*}
  satisfy the same recurrence with matching initial values. This can be done
  algorithmically using, for instance, creative telescoping. In practice, the
  fact that $B_{\sigma} (n)$ is a triple sum makes the computation of the
  recurrence rather challenging. Yet, Koutschan's Mathematica package
  \texttt{HolonomicFunctions} \cite{koutschan-phd} is able to determine a
  fourth order linear recurrence for $B_{\sigma} (n)$ (see below for more
  information on this recurrence). We then verify that this recurrence is a
  right factor of the earlier recurrence of order $12$ for $A_{\sigma} (n)$.
  Since the first $12$ initial values match (in fact, additional reflection
  shows that two matching initial values suffice), we conclude that
  $A_{\sigma} (n) = B_{\sigma} (n)$.
\end{proof}

The proof of Proposition~\ref{prop:A86} demonstrates that any individual evaluation of leading
coefficients in terms of binomial sums can, in principle, be algorithmically proven due to recent advances
in symbolic computation. It is curious to note that we had to use Maple, Magma and Mathematica in that computation.

\begin{remark}
We find that $A_{\sigma_8} (n)$, with $\sigma_8$ as in
Proposition~\ref{prop:A86}, is the unique solution of a fourth order
recurrence
\begin{equation*}
  p_4 (n) a_{n + 4} + p_3 (n) a_{n + 3} + p_2 (n) a_{n + 2} + p_1 (n) a_{n +
   1} + p_0 (n) a_n = 0,
\end{equation*}
with initial conditions $a_0 = 1, a_1 = 33$ and $a_j = 0$ for $j < 0$. Here,
the coefficients $p_j (n)$ are polynomials of degree $15$, satisfying
\begin{equation*}
  p_j (n) = - p_{4 - j} (- 5 - n),
\end{equation*}
for $j \in \{ 1, 2, 3, 4 \}$. The latter relation is a consequence of the fact
that the configuration $\sigma_8$ is self-dual (see \cite[Section~4]{brown-apery}).
\end{remark}

\subsection{Finding multiple binomial sum representations for $A_{\sigma} (n)$}

In this section, we illustrate how binomial sum representations can be found for any convergent configuration $\sigma$.
As in the previous section, we proceed with the configuration $\sigma = \sigma_8 = (8, 3, 6, 1, 4, 7, 2, 5)$.
Consider the cellular integral
\begin{equation*}
  I_{\sigma} (n) = \int_{S_8} f_{\sigma}^n \,\, \omega_{\sigma},
\end{equation*}
where $f_{\sigma}$ is as in \eqref{eq:f:z:86}. Because the action of
$\operatorname{PGL}_2$ is triply transitive, we may also make the convenient choice
$z_{\sigma (6)} = z_7 = 1$, $z_{\sigma (7)} = z_2 = 0$, $z_{\sigma (8)} = z_5
= \infty$. Then, in the coordinates $t_1 = z_1, t_2 = z_3, t_3 = z_4, t_4 =
z_6, t_5 = z_8$, we have
\begin{equation*}
  f_{\sigma} = \frac{(t_1) (- t_2) (t_2 - t_3) (t_4 - 1) (1 - t_5) (t_5 -
   t_1)}{(t_5 - t_2) (t_2 - t_4) (t_4 - t_1) (t_1 - t_3) (t_3 - 1)}
\end{equation*}
and
\begin{equation*}
  \omega_{\sigma} = \frac{dt_1 dt_2 dt_3 dt_4 dt_5}{(t_5 - t_2) (t_2 - t_4) (t_4 - t_1) (t_1 - t_3) (t_3 - 1)} .
\end{equation*}
We then substitute
\begin{equation*}
  x_1 = t_5 - t_2, \quad x_2 = t_2 - t_4, \quad x_3 = t_4 - t_1, \quad x_4 =
   t_1 - t_3, \quad x_5 = t_3 - 1.
\end{equation*}
Observe that $t_3 = x_5 + 1$, $t_1 = x_4 + x_5 + 1$ and, likewise,
\begin{equation*}
  t_4 = x_3 + x_4 + x_5 + 1, \quad t_2 = x_2 + x_3 + x_4 + x_5 + 1, \quad t_5
   = x_1 + x_2 + x_3 + x_4 + x_5 + 1,
\end{equation*}
so that, up to a sign in $\omega_{\sigma}$,
\begin{eqnarray*}
  f_{\sigma} & = & \frac{x_{4, 6} x_{2, 6} x_{2, 4} x_{3, 5} x_{1, 5} x_{1,
  3}}{x_1 x_2 x_3 x_4 x_5},\\
  \omega_{\sigma} & = & \frac{dx_1 dx_2 dx_3 dx_4 d x_5}{x_1 x_2 x_3 x_4 x_5},
\end{eqnarray*}
where $x_{i, j} = x_i + x_{i + 1} + \ldots + x_j$ with $x_6 = 1$.

At this point, it is natural to also consider the integral
\begin{equation*}
  J_{\sigma} (n) = \frac{1}{(2\pi i)^5} \oint_{| x_i | = \varepsilon} f_{\sigma}^n \omega_{\sigma}
   = \frac{1}{(2\pi i)^5} \oint_{| x_i | = \varepsilon} \left(\frac{x_{4, 6} x_{2, 6} x_{2, 4}
   x_{3, 5} x_{1, 5} x_{1, 3}}{x_1 x_2 x_3 x_4 x_5} \right)^n \frac{d x_1 d x_2 d x_3 d x_4 d x_5}{x_1 x_2 x_3 x_4 x_5},
\end{equation*}
where $\varepsilon$ is chosen sufficiently small so that the integrals
converge. By the residue theorem, this integral evaluates to
\begin{equation*}
  J_{\sigma} (n) = [(x_1 x_2 x_3 x_4 x_5)^n] (x_{4, 6} x_{2, 6} x_{2, 4}
   x_{3, 5} x_{1, 5} x_{1, 3})^n .
\end{equation*}
In particular, the values $J_{\sigma} (n)$ are nonnegative integers. By the
principle of creative telescoping, we can derive a linear recurrence which is
satisfied by both $I_{\sigma} (n)$ and $J_{\sigma} (n)$.
In fact, we are going to show that the leading coefficients $A_{\sigma} (n)$ of
$I_{\sigma} (n)$ are equal to $J_{\sigma} (n)$. Likely, one can prove
this equality in a general uniform fashion (for instance, following the approach
of \cite{dupont-odd}, as suggested by Dupont).
For our purposes, it suffices to observe that both sequences satisfy a common
recurrence and agree to sufficiently many terms. It remains to express
$J_{\sigma} (n)$ in terms of a multiple binomial sum.

Here, it will be convenient to not specialize $x_6$. In order to find an
explicit formula for $J_{\sigma} (n)$, we expand
\begin{equation*}
  \Lambda = (x_{1, 5} x_{2, 6} x_{1, 3} x_{2, 4} x_{3, 5} x_{4, 6})^n
\end{equation*}
using the binomial theorem and then extract the coefficient of $(x_1 x_2 x_3
x_4 x_5 x_6)^n$. Some care needs to be applied at this stage, because the
order in which terms are expanded can have a considerable influence on the
final binomial sum. For instance, the number of summations can vary
substantially. In the present case, it is natural to first expand
\begin{equation*}
  x_{1, 5}^n = \sum_{k_1 = 0}^n \binom{n}{k_1} x_{1, 3}^{k_1} x_{4, 5}^{n -
   k_1}, \quad x_{2, 6} = \sum_{k_2 = 0}^n \binom{n}{k_2} x_{4, 6}^{k_2} x_{2,
   3}^{n - k_2},
\end{equation*}
and, in a second step, $x_{1, 3}^{n + k_1}$ as well as $x_{4, 6}^{n + k_4}$,
so that $\Lambda$ equals
\begin{eqnarray*}
  &  & \sum_{k_1, k_4} \binom{n}{k_1} \binom{n}{k_4} x_{1, 3}^{n + k_1} x_{2,
  4}^n x_{3, 5}^n x_{4, 6}^{n + k_4} x_{2, 3}^{n - k_4} x_{4, 5}^{n - k_1}\\
  & = & \sum_{k_1, k_4, \ell_1, \ell_4} \binom{n}{k_1} \binom{n}{k_4}
  \binom{n + k_1}{\ell_1} \binom{n + k_4}{\ell_4} x_{2, 4}^n x_{3, 5}^n x_{2,
  3}^{n - k_4 + \ell_1} x_{4, 5}^{n - k_1 + \ell_4} x_1^{n + k_1 - \ell_1}
  x_6^{n + k_4 - \ell_4}.
\end{eqnarray*}
Since $J_{\sigma} (n)$ is the coefficient of $(x_1 x_2 x_3 x_4 x_5 x_6)^n$ in
$\Lambda$, we conclude that $J_{\sigma} (n)$ is the coefficient of $(x_2 x_3
x_4 x_5)^n$ in
\begin{eqnarray}
  &  & \sum_{k_1, k_4} \binom{n}{k_1} \binom{n}{k_4} \binom{n + k_1}{k_1}
  \binom{n + k_4}{k_4} x_{2, 4}^n x_{3, 5}^n x_{2, 3}^{n - k_4 + k_1} x_{4,
  5}^{n - k_1 + k_4} .  \label{eq:A86:14}
\end{eqnarray}
We next expand $x_{2, 4}^n$ and $x_{3, 5}^n$ to obtain
\begin{eqnarray*}
  &  & x_{2, 4}^n x_{3, 5}^n x_{2, 3}^{n - k_4 + k_1} x_{4, 5}^{n - k_1 +
  k_4}\\
  & = & \sum_{k_2, k_3} \binom{n}{k_2} \binom{n}{k_3} x_{2, 3}^{n - k_4 + k_1
  + k_2} x_{4, 5}^{n - k_1 + k_4 + k_3} x_3^{n - k_3} x_4^{n - k_2}\\
  & = & \sum_{k_2, k_3, \ell_2, \ell_3} \binom{n}{k_2} \binom{n}{k_3}
  \binom{n - k_4 + k_1 + k_2}{\ell_2} \binom{n - k_1 + k_4 + k_3}{\ell_3}\\
  &  & \qquad \times \,\, x_2^{\ell_2} x_3^{n - k_3 + n - k_4 + k_1 + k_2 - \ell_2}
  x_4^{n - k_2 + n - k_1 + k_4 + k_3 - \ell_3} x_5^{\ell_3}.
\end{eqnarray*}
The coefficient of $(x_2 x_3 x_4 x_5)^n$ in that sum is
\begin{equation*}
  \sum_{k_2, k_3} \binom{n}{k_2} \binom{n}{k_3} \binom{n + k_2}{n} \binom{n +
   k_3}{n},
\end{equation*}
subject to the constraint $k_1 + k_2 = k_3 + k_4$. Finally, combined with
\eqref{eq:A86:14}, we conclude that
\begin{equation*}
  J_{\sigma} (n) = \sum_{\substack{
     k_1, k_2, k_3, k_4 =0 \\
     k_1 + k_2 = k_3 + k_4
   }}^{n} \prod_{i = 1}^4 \binom{n}{k_i} \binom{n + k_i}{n},
\end{equation*}
which is the binomial sum of Proposition~\ref{prop:A86}. A similar procedure has been carried out to find explicit binomial expressions for the remaining 897 leading coefficients $A_{\sigma_N}(n)$ for $7 \leq N \leq 10$. Again, these expressions are available upon request.

\section{Proof of Theorem \ref{main2}} \label{sec:proof2}

\subsection{Congruences for binomial coefficients and harmonic sums}

For a nonnegative integer $n$, we define the harmonic sum ${H}_{n}$ by
\begin{equation*}
  {H}_{n} \assign \sum^{n}_{j=1} \frac{1}{j}
\end{equation*}
and ${H}_{0} \assign 0$.
We first recall some elementary congruences (see Section 7.7, Theorems 133, 132 and 116 in \cite{HW}): For $p$ an odd prime, we have

\begin{equation} \label{p1}
{\left(\frac{p-1}{2}\right)!}^4 \equiv 1 \pmod{p},
\end{equation}

\begin{equation} \label{p3}
\binom{p-1}{\frac{p-1}{2}} \equiv (-1)^{\frac{p-1}{2}} \, 2^{2p-2} \pmod{p^2},
\end{equation}

\begin{equation} \label{p4}
2^{p-1}-1 \equiv p \left( 1 + \frac{1}{3} + \frac{1}{5} \cdots + \frac{1}{p-2}\right) \pmod{p^2}
\end{equation}
and, for primes $p > 3$, 
\begin{equation} \label{p2}
H_{p-1} \equiv 0 \pmod{p^2}.
\end{equation}

We will also need the following result, which follows easily from \eqref{p4} and \eqref{p2}.

\begin{lemma}\label{cor_2powerp2}
For a prime $p>3$,
$$2^{2p-2} \equiv 1 - p H_{\frac{p-1}{2}} \pmod{p^2}.$$
\end{lemma}

For nonnegative integers $n$, let $\ph{a}{n} \assign a(a+1)(a+2)\dotsm(a+n-1)$ denote the rising factorial, with $\ph{a}{0}:=1$.
Let $p$ be an odd prime. We note that, for $0 \leq k \leq (p-1)/2$,
\begin{align*}
  \binom{\frac{p-1}{2}+k}{k}
  \equiv
  \frac{(1+k)(2+k) \cdots (\frac{p-1}{2} + k)}{\fac{\left(\frac{p-1}{2}\right)}}
  \equiv
  \frac{(k+1)_{\frac{p-1}{2}}}{\fac{\left(\frac{p-1}{2}\right)}}
  \pmod{p},
\end{align*}
and, similarly,
\begin{equation}\label{eq:Bin_to_Poch}
\binom{\frac{p-1}{2}}{k} \binom{\frac{p-1}{2}+k}{k}
\equiv
(-1)^k \, \left(\frac{(k+1)_{\frac{p-1}{2}}}{\left(\frac{p-1}{2}\right)!}\right)^2
\pmod{p},
\end{equation}
as well as
\begin{align}\label{eq:bin2:p}
  \binom{p-1-k}{\frac{p-1}{2}}
  &\equiv
  (-1)^{\frac{p-1}{2}} \,
  \binom{\frac{p-1}{2}+k}{k}
  \pmod{p}.
\end{align}

\begin{lemma}\label{lem_BinHar2}
For $p$ an odd prime,
\begin{equation*}
\sum_{k=0}^{p-1} (k+1)_{\frac{p-1}{2}}^2 \equiv -1 \pmod{p}.
\end{equation*}
\end{lemma}

\begin{proof}
Let $m \assign \frac{p-1}{2}$ and write
\begin{equation*}
(k+1)_{m}^2
= b_0+  \sum_{t=1}^{2m} b_t k^t
\end{equation*}
for appropriate integers $b_t$.
Recall that, for positive integers $s$,
\begin{equation}\label{exp_sums}
\sum^{p-1}_{j=1} j^s\equiv
\begin{cases}
-1 \pmod {p},& \text{if $(p-1) \vert s$} \, ,\\
\phantom{-}0 \pmod {p},& \text{otherwise} \, .
\end{cases}
\end{equation}
Therefore,
\begin{equation*}
\sum_{k=0}^{p-1} (k+1)_{m}^2
= p \, b_0 + \sum_{k=0}^{p-1} \sum_{t=1}^{2m} b_t k^t
= p \, b_0 +  \sum_{t=1}^{2m} b_t \sum_{k=1}^{p-1} k^t
\equiv
- b_{p-1}
\pmod{p}.
\end{equation*}
Since $(k+1)_{m}^2$ is monic, $b_{p-1}=1$.
\end{proof}

\begin{lemma}\label{lem_BinHar1}
For $p$ an odd prime,
\begin{equation}\label{eq:BinHar1}
\sum_{\substack{k_1, k_2, k_3, k_4=0\\ \sum k_i  = p-1}}^{\frac{p-1}{2}} 
\left[
\prod_{i=1}^{4}
{(k_i+1)}_{\frac{p-1}{2}}^2
\right]
\left(
H_{\frac{p-1}{2}+k_4}-H_{k_4}
\right)
\equiv 0
\pmod{p}.
\end{equation}
\end{lemma}

\begin{proof}
For brevity, we again write $m=\frac{p-1}{2}$. Let $0 \leq k \leq \frac{p-1}{2}$ and consider
\begin{align*}
(m-k+1)_m
&= (\tfrac{p-1}{2}-k+1) (\tfrac{p-1}{2}-k+2) \cdots (p-2-k) (p-1-k)\\
&\equiv (-\tfrac{p-1}{2}-k) (-\tfrac{p-1}{2}+1-k) \cdots (-2-k) (-1-k)\\
&\equiv (-1)^m (k+1)_m
\pmod{p}.
\end{align*}
Similarly,
\begin{align*}
  H_{2m-k}-H_{m-k}
  &=\frac{1}{p-1-k}+\frac{1}{p-2-k} + \cdots + \frac{1}{\frac{p-1}{2}-k+2} +\frac{1}{\frac{p-1}{2}-k+1}\\
  &\equiv \frac{1}{-1-k}+\frac{1}{-2-k} + \cdots + \frac{1}{-\frac{p-1}{2}-k+1} +\frac{1}{-\frac{p-1}{2}-k} \\
  &\equiv - \left(H_{m+k}-H_{k}\right)
  \pmod{p}.
\end{align*}
Let
\begin{align*}
  S = 
  \sum_{\substack{k_1, k_2, k_3, k_4=0\\ \sum k_i  = p-1}}^{m} 
  \left[
  \prod_{i=1}^{4}
  {(k_i+1)}_{m}^2
  \right]
  \left(
  H_{m+k_4}-H_{k_4}
  \right)
\end{align*}
be the left-hand side of \eqref{eq:BinHar1}.
By replacing $k_i$ with $m-k_i$, for $i\in\{1,2,3,4\}$, we see that
\begin{align*}
  S =
  \sum_{\substack{k_1, k_2, k_3, k_4=0\\ \sum k_i  = p-1}}^{m} 
  \left[ \prod_{i=1}^{4} {(m-k_i+1)}_{m}^2 \right] \left( H_{2m - k_4}-H_{m-k_4} \right)
  \equiv -S \pmod{p}.
\end{align*}
Hence, the sum $S$ must vanish modulo $p$.
\end{proof}

\begin{lemma}\label{lem_BinHar3}
For $p$ an odd prime,
\begin{equation*}
  \sum_{\substack{k_1, k_2, k_3, k_4=0\\ \sum k_i  = p-1}}^{\frac{p-1}{2}} \,
  \prod_{i=1}^{4}
  \binom{\frac{p-1}{2}}{k_i} \binom{\frac{p-1}{2}+k_i}{k_i}
  \equiv
  \sum_{\substack{k_1,k_2, k_3, k_4 = 0\\k_1+k_2=k_3+k_4}}^{\frac{p-1}{2}}
  \prod_{i=1}^{4}
  \binom{\frac{p-1}{2}}{k_i} \binom{\frac{p-1}{2}+k_i}{k_i}
  \pmod{p^2}.
\end{equation*}
\end{lemma}

\begin{proof}
Once more, $m:= \frac{p-1}{2}$. Replacing $k_1$ and $k_2$ with $m-k_1$ and $m-k_2$, respectively, yields
\begin{equation*}
  \sum_{\substack{k_1, k_2, k_3, k_4=0\\ \sum k_i  = p-1}}^{m} \,
  \prod_{i=1}^{4}
  \binom{m}{k_i} \binom{m+k_i}{k_i}
  =
  \sum_{\substack{k_1,k_2, k_3, k_4 = 0\\k_1+k_2=k_3+k_4}}^{m}
  \prod_{i=1}^{2}
  \binom{m}{k_i}  \binom{2m-k_i}{m - k_i}
  \prod_{i=3}^{4}
  \binom{m}{k_i} \binom{m+k_i}{k_i}.
\end{equation*}
Let $\underline{k} \assign (k_1,k_2,k_3,k_4)$ and define
\begin{equation*}
  f(p,\underline{k}) \assign 
  \prod_{i=1}^{2}
  \binom{m}{k_i}  \binom{2m-k_i}{m - k_i}
  \prod_{i=3}^{4}
  \binom{m}{k_i} \binom{m+k_i}{k_i}
\end{equation*}
and
\begin{equation*}
  g(p,\underline{k}) \assign 
  \prod_{i=1}^{4}
  \binom{m}{k_i} \binom{m+k_i}{k_i}.
\end{equation*}
Then it suffices to prove
\begin{equation*}
  \sum_{\substack{k_1,k_2, k_3, k_4 = 0\\k_1+k_2=k_3+k_4}}^{m}
  f(p,\underline{k})
  \equiv
  \sum_{\substack{k_1,k_2, k_3, k_4 = 0\\k_1+k_2=k_3+k_4}}^{m}
  g(p,\underline{k})
  \pmod{p^2}.
\end{equation*}
It follows from \eqref{eq:bin2:p} that
\begin{align*}
  \binom{2m-k}{m - k} = \binom{2m-k}{m}
  &\equiv
  (-1)^{m} \, \binom{m+k}{k}
  \pmod{p}.
\end{align*}
Consequently, letting $m-\underline{k} \assign (m-k_1,m-k_2,m-k_3,m-k_4)$, we have that
\begin{align*}
  f(p,\underline{k})& + f(p,m-\underline{k}) - g(p,\underline{k}) - g(p,m-\underline{k})\\
  &=
  \left[ \prod_{i=1}^{4} \binom{m}{k_i} \right]
  \left[ \prod_{i=1}^{2}  \binom{2m-k_i}{m - k_i} - \prod_{i=1}^{2} \binom{m+k_i}{k_i} \right]
  \left[ \prod_{i=3}^{4}  \binom{m+k_i}{k_i} - \prod_{i=3}^{4}  \binom{2m-k_i}{m - k_i} \right]\\
  &\equiv 0 \pmod{p^2}.
\end{align*}
Therefore,
\begin{align*}
  2  \sum_{\substack{k_1,k_2, k_3, k_4 = 0\\k_1+k_2=k_3+k_4}}^{m} 
  \left[ f(p,\underline{k}) - g(p,\underline{k}) \right]
  &=
  \sum_{\substack{k_1,k_2, k_3, k_4 = 0\\k_1+k_2=k_3+k_4}}^{m}
  \left[ f(p,\underline{k}) + f(p,m-\underline{k}) - g(p,\underline{k}) - g(p,m-\underline{k}) \right]\\
  &\equiv 0 \pmod{p^2}.
\end{align*}
\end{proof}

\subsection{Multiplicative characters and finite field hypergeometric functions}
Let $\widehat{\mathbb{F}^{*}_{p}}$ denote the group of multiplicative characters of $\mathbb{F}^{*}_{p}$. 
We extend the domain of $\chi \in \widehat{\mathbb{F}^{*}_{p}}$ to $\mathbb{F}_{p}$, by defining $\chi(0) \assign 0$ (including the trivial character $\varepsilon_p$) and denote $\bar{\chi}$ as the inverse of $\chi$. When $p$ is odd we denote the character of order 2 of $\mathbb{F}_p^*$ by $\phi_p$. We will drop the subscript $p$ if it is clear from the context.
We recall the following orthogonality relation. 
For $\chi \in \widehat{\mathbb{F}_p^{*}}$, we have
\begin{equation}\label{for_TOrthEl}
  \sum_{x \in \mathbb{F}_p} \chi(x)=
  \begin{cases}
  p-1, & \text{if $\chi = \varepsilon$}  ,\\
  0, & \text{if $\chi \neq \varepsilon$}  .
  \end{cases}
\end{equation}
For $A$, $B  \in \widehat{\mathbb{F}^{*}_{p}}$, let
\begin{align*}
  \bin{A}{B} \assign \frac{B(-1)}{p} \sum_{x \in \mathbb{F}_{p}} A(x) \bar{B}(1-x).
\end{align*}
Then, for $A_0,A_1,\dotsc, A_n$, $B_1, \dotsc, B_n  \in \widehat{\mathbb{F}^{*}_{p}}$ and $x \in \mathbb{F}_{p}$, the finite field hypergeometric function of Greene \cite{G} is defined as 
\begin{equation*}
  {_{n+1}F_{n}} {\biggl( \begin{array}{cccc} A_0, & A_1, & \dotsc, & A_n \\
   \phantom{A_0,} & B_1, & \dotsc, & B_n \end{array}
  \Big| \; x \biggr)}_{p}
  \assign \frac{p}{p-1} \sum_{\chi  \in \widehat{\mathbb{F}^{*}_{p}}} \binom{A_0 \chi}{\chi} \prod_{i=1}^{n} \binom{A_i \chi}{B_i \chi} \chi(x) .
\end{equation*}
We consider the case where $A_i = \phi_p$ for all $i$ and $B_j= \varepsilon_p$ for all $j$ and write
\begin{equation*}
  {_{n+1}F_n}(x) =
  {_{n+1}F_n} \biggl( \begin{matrix} \phi_p,\, \phi_p,\, \dots,\, \phi_p \\
  \varepsilon_p,\, \dots,\, \varepsilon_p \end{matrix} \biggm| x \biggr)_{p}
\end{equation*}
for brevity.

Let $\mathbb{Z}_p$ denote the ring of $p$-adic integers and $\mathbb{Z}^{*}_p$ its group of units.
We define the Teichm\"{u}ller character to be the primitive character $\omega: \mathbb{F}_p \rightarrow\mathbb{Z}^{*}_p$ satisfying $\omega(x) \equiv x \pmod p$ for all $x \in \{0,1, \ldots, p-1\}$. In fact,
\begin{equation}\label{for_Teich}
\omega(x) \equiv x^{p^{n-1}} \pmod {p^n} 
\end{equation}
for all $n \geq 1$.

Theorem \ref{main2} is now implied by Proposition~\ref{prop:A86} and the following result.

\begin{theorem}\label{thm_Main2}
Let 
$$
\eta^{12}(2z) =: \sum_{n=1}^{\infty} \gamma(n) q^n
$$
be the unique newform in $S_6(\Gamma_0(4))$.
Then, for $p$ an odd prime,
\begin{equation} \label{n8p2}
  \gamma(p) \equiv 
  \sum_{\substack{k_1,k_2, k_3, k_4 = 0\\k_1+k_2=k_3+k_4}}^{\frac{p-1}{2}}
  \prod_{i=1}^{4}
  \binom{\frac{p-1}{2}}{k_i} \binom{\frac{p-1}{2}+k_i}{k_i}
  \pmod{p^2}.
\end{equation}
\end{theorem}

\begin{proof}
We confirm that \eqref{n8p2} holds for $p=3$ and assume $p\geq 5$ henceforth.
For the Legendre family of elliptic curves $E_{\lambda}$, given by
\begin{equation*}\label{def_ELambda}
  E_{\lambda}: y^2=x(x-1)(x-\lambda), \qquad \lambda \not\in  \lbrace 0, 1 \rbrace,
\end{equation*}
we define
\begin{equation*}\label{def_ap}
  a(p, \lambda) \assign p+1-\# E_{\lambda}(\mathbb{F}_p).
\end{equation*}
Then, from \cite[Proposition~2.1]{FOP} and using the fact that $\dim(S_6(\Gamma_0(4)))=1$, we get that for $p$ an odd prime,
\begin{align*}
  \gamma(p) 
  &= -3 - \sum_{\lambda=2}^{p-1} \left( a(p,\lambda)^4 - 3 p \, a(p,\lambda)^2 +p^2 \right)\\
  &= -p^3 +2p^2 -3 - \sum_{\lambda=2}^{p-1} \left( a(p,\lambda)^4 - 3 p \, a(p,\lambda)^2\right).
\end{align*}
From \cite[Lemma~2.1]{A2} we have
\begin{equation*}
\sum_{\lambda=2}^{p-1} a(p,\lambda)^2 = p^2 - 2p - 3,
\end{equation*}
so that, combining these two equations,
\begin{align}\label{for_dp1}
  \gamma(p) 
  &= 2p^3 -4p^2 -9p -3 - \sum_{\lambda=2}^{p-1}  a(p,\lambda)^4.
\end{align}
In \cite[Section~4]{Kk}, Koike showed that for $\lambda \neq 0,1$,
\begin{equation*}
  a(p,\lambda) = -\phi(-1) \cdot p \cdot {_{2}F_1}(\lambda).
\end{equation*}
We now recall a couple of facts about finite field hypergeometric functions from \cite[Proposition~3]{O} and \cite[Theorem~4.2]{G} respectively:
\begin{equation*}
  p \cdot {_{2}F_1}(1) 
  =-\phi(-1),
\end{equation*}
and, for $\lambda \neq 0$,
\begin{equation*}
  {_{2}F_1}(\lambda) 
  =  \phi(\lambda) \cdot  {_{2}F_1}\Bigl(\frac{1}{\lambda}\Bigr) .
\end{equation*}
Therefore,
\begin{align}\label{for_ap1}
  \sum_{\lambda=2}^{p-1}  a(p,\lambda)^4
  \notag &= 
  \sum_{\lambda=2}^{p-1}
  \left[
  -\phi(-1) \cdot p \cdot {_{2}F_1}(\lambda) 
  \right]^4\\
  \notag &=
  \sum_{\lambda=1}^{p-1}
  \left[
  -\phi(-1) \cdot p \cdot {_{2}F_1}(\lambda) 
  \right]^4
  -
  \left[
  -\phi(-1) \cdot p \cdot {_{2}F_1}(1) 
  \right]^4\\
  \notag &=
  \sum_{\lambda=1}^{p-1}
  \left[
  -\phi(-1) \cdot p \cdot {_{2}F_1}(\lambda) 
  \right]^4
  -1\\
  &=
  \sum_{\lambda=1}^{p-1}
  \left[
  -\phi(-\lambda) \cdot p \cdot {_{2}F_1}\Bigl(\frac{1}{\lambda}\Bigr) 
  \right]^4
  -1.
\end{align}
By Theorem 2.4 in \cite{LR} (see also Theorem 1.1 in \cite{os}) and \eqref{for_Teich}, we get that, for $\lambda \neq 0$, 
\begin{align}\label{for_LR}
\notag -&\phi(-\lambda) \cdot p \cdot {_{2}F_1}\Bigl(\frac{1}{\lambda}\Bigr) \\ 
&=
\notag -\phi(-\tfrac{1}{\lambda}) \cdot p \cdot
{_{2}F_1}\Bigl(\frac{1}{\lambda}\Bigr) \\ 
\notag &\equiv
(p+1) \sum_{j=0}^{(p-1)/2}
\binom{\frac{p-1}{2}}{j} \binom{\frac{p-1}{2}+j}{j}
(-1)^j \lambda^{jp}
\left(
1 + 2jp 
\left(
H_{\frac{p-1}{2}+j}-H_j
\right)
\right)\\
&\equiv
(p+1) \sum_{j=0}^{(p-1)/2}
\binom{\frac{p-1}{2}}{j} \binom{\frac{p-1}{2}+j}{j}
(-1)^j
\left(
1 + 2jp 
\left(
H_{\frac{p-1}{2}+j}-H_j
\right)
\right)
\omega(\lambda^{j})
\pmod{p^2}.
\end{align}
For brevity, we again let $m=(p-1)/2$ and write
\begin{equation*}
f(p, k_1,k_2,k_3,k_4) \assign 
\prod_{i=1}^{4}
\binom{m}{k_i} \binom{m+k_i}{k_i}
(-1)^{k_i}
\left(
1 + 2p k_i 
\left(
H_{m+k_i}-H_{k_i}
\right)
\right).
\end{equation*}
Combining \eqref{for_ap1} and \eqref{for_LR} yields
\begin{align}\label{for_ap2}
\notag 1 + \sum_{\lambda=2}^{p-1}  a(p,\lambda)^4
\notag & \equiv
\sum_{\lambda=1}^{p-1}
\left[
(p+1) \sum_{j=0}^{m}
\binom{m}{j} \binom{m+j}{j}
(-1)^j
\left(
1 + 2jp 
\left(
H_{m+j}-H_j
\right)
\right)
\omega^{j}(\lambda)
\right]^4\\
\notag & \equiv 
(4p+1) 
\sum_{k_1,k_2,k_3,k_4=0}^{m}
f(p, k_1,k_2,k_3,k_4)
\sum_{\lambda=1}^{p-1} \omega^{k_1+k_2+k_3+k_4}(\lambda)\\
& \equiv 
-(3p+1) 
\sum_{\substack{k_1, k_2, k_3, k_4=0\\ \sum k_i  \equiv \, 0 \, \, (\text{mod} \,\, p-1)}}^{m}
f(p, k_1,k_2,k_3,k_4)
\pmod{p^2},
\end{align}
where we have used \eqref{for_TOrthEl} and the facts that $(p+1)^4 \equiv 4p+1 \pmod{p^2}$ and $(4p+1)(p-1) \equiv -(3p+1) \pmod{p^2}$.
Accounting for \eqref{for_ap2} in \eqref{for_dp1}, we now have
\begin{equation*}\label{for_dp2}
  \gamma(p) 
  \equiv -9p-2 +(3p+1) 
  \sum_{\substack{k_1, k_2, k_3, k_4=0\\ \sum k_i  \equiv  \, 0 \, \, (\text{mod} \,\, p-1)}}^{m}
  f(p, k_1,k_2,k_3,k_4)
  \pmod{p^2}.
\end{equation*}

In order that $\sum k_i \equiv 0 \pmod{p-1}$, the sum $\sum k_i$ must be either $0$, $p-1$ or $2(p-1)$.
In the case $\sum k_i=0$, we necessarily have $k_1=k_2=k_3=k_4=0$, which contributes $f(p,0,0,0,0) = 1$.
In the third case, that is $\sum k_i=2(p-1)$, we necessarily have $k_1=k_2=k_3=k_4=\frac{p-1}{2}=m$.
Applying \eqref{p3}, \eqref{p2} and Lemma~\ref{cor_2powerp2}, we evaluate
\begin{align*}
f\left(p,m,m,m,m\right)
&=
\left[
\binom{m}{m} \binom{p-1}{m}
(-1)^{m}
\left(
1 + p(p-1) 
\left(
H_{p-1}-H_{m}
\right)
\right)
\right]^4\\
&\equiv
\left[
\binom{p-1}{m}
\left(
1 - p
\left(
H_{p-1}-H_{m}
\right)
\right)
\right]^4\\
&\equiv
\left[
(-1)^{m} \, 2^{2p-2}
\left(
1 + p \,
H_{m}
\right)
\right]^4\\
&\equiv
\left[
(-1)^{m} \, 
\left(
1 - p \,
H_{m}
\right)
\left(
1 + p \,
H_{m}
\right)
\right]^4
\equiv
\left[
(-1)^{m} \, 
\right]^4
\equiv
1
\pmod{p^2}.
\end{align*}
The only other possibility is $\sum k_i=p-1$.
Therefore,
\begin{equation*}\label{for_dp3}
\gamma(p) 
\equiv -3p +(3p+1) 
\sum_{\substack{k_1, k_2, k_3, k_4=0\\ \sum k_i  \, = \, p-1}}^{m}
f(p, k_1,k_2,k_3,k_4)
\pmod{p^2}.
\end{equation*}
Noting that $H_{m+j}-H_j \in \mathbb{Z}_p$ for $0 \leq j \leq m$, we have
\begin{align*}
f(p, k_1,k_2,k_3,k_4)
&=
\prod_{i=1}^{4}
\binom{m}{k_i} \binom{m+k_i}{k_i}
(-1)^{k_i}
\left(
1 + 2p k_i 
\left(
H_{m+k_i}-H_{k_i}
\right)
\right)\\
&\equiv
\left[
\prod_{i=1}^{4}
\binom{m}{k_i} \binom{m+k_i}{k_i}
\right]
\left(
1 + 2p \sum_{i=1}^{4} k_i 
\left(
H_{m+k_i}-H_{k_i}
\right)
\right)
\pmod{p^2},
\end{align*}
and so 
\begin{multline}\label{for_dp4}
\gamma(p) 
\equiv 
-3p + p
\sum_{\substack{k_1, k_2, k_3, k_4=0\\ \sum k_i  \, = \, p-1}}^{m} 
g(p, k_1,k_2,k_3,k_4)
+\sum_{\substack{k_1, k_2, k_3, k_4=0\\ \sum k_i  \, = \, p-1}}^{m} \,
\prod_{i=1}^{4}
\binom{m}{k_i} \binom{m+k_i}{k_i}
\pmod{p^2},
\end{multline}
where, for brevity,
\begin{equation*}
g(p, k_1,k_2,k_3,k_4)  \assign 
\left[
\prod_{i=1}^{4}
\binom{m}{k_i} \binom{m+k_i}{k_i}
\right]
\left(
3 + 2 \sum_{i=1}^{4} k_i 
\left(
H_{m+k_i}-H_{k_i}
\right)
\right).
\end{equation*}
We will now show that
\begin{equation}\label{for_g_1}
\sum_{\substack{k_1, k_2, k_3, k_4=0\\ \sum k_i  \, = \, p-1}}^{m} 
g(p, k_1,k_2,k_3,k_4) 
\equiv 3
\pmod{p}.
\end{equation}
Using \eqref{eq:Bin_to_Poch} and \eqref{p1}, we note that, when $\sum k_i=p-1$,
\begin{align}\label{for_g_2}
\notag
g(p, k_1,k_2,k_3,k_4) 
&\equiv
\left[
\prod_{i=1}^{4}
(-1)^{k_i}
\left(\frac{{(k_i+1)}_{m}}{\fac{\left(m\right)}} \right)^2
\right]
\left(
3 + 2 \sum_{i=1}^{4} k_i 
\left(
H_{m+k_i}-H_{k_i}
\right)
\right)\\
&\equiv
\left[
\prod_{i=1}^{4}
{(k_i+1)}_{m}^2
\right]
\left(
3 + 2 \sum_{i=1}^{4} k_i 
\left(
H_{m+k_i}-H_{k_i}
\right)
\right)
\pmod{p}.
\end{align}
Let us define
\begin{multline}\label{def_g1}
g_1(p, k_1,k_2,k_3,k_4) \\ \assign 
\left[
\prod_{i=1}^{4}
{(k_i+1)}_{m}^2
\right]
\left(
3 + 2 \sum_{i=1}^{3} k_i 
\left(
H_{m+k_i}-H_{k_i}
\right)
-2(k_1+k_2+k_3)
\left(
H_{m+k_4}-H_{k_4}
\right)
\right)
\end{multline}
as well as
\begin{equation*}
  g_2(p, k_1,k_2,k_3,k_4) \\ \assign 
  \left[
  \prod_{i=1}^{4}
  {(k_i+1)}_{m}^2
  \right]
  2(p-1)
  \left(
  H_{m+k_4}-H_{k_4}
  \right).
\end{equation*}
Then, using \eqref{for_g_2} and Lemma \ref{lem_BinHar1}, we get that 
\begin{align*}
  \sum_{\substack{k_1, k_2, k_3, k_4=0\\ \sum k_i  \, = \, p-1}}^{m} 
  g(p, k_1,k_2,k_3,k_4) 
  &\equiv
  \sum_{\substack{k_1, k_2, k_3, k_4=0\\ \sum k_i  \, = \, p-1}}^{m} 
  g_1(p, k_1,k_2,k_3,k_4) 
  +\sum_{\substack{k_1, k_2, k_3, k_4=0\\ \sum k_i \, = \, p-1}}^{m} 
  g_2(p, k_1,k_2,k_3,k_4) \\
  &\equiv
  \sum_{\substack{k_1, k_2, k_3, k_4=0\\ \sum k_i \, = \, p-1}}^{m} 
  g_1(p, k_1,k_2,k_3,k_4) 
  \pmod{p}.
\end{align*}
If $m < k < p$, then $H_{m+k} - H_k \in \frac{1}{p} \mathbb{Z}_p$ and ${(k+1)}_{m} \in p \mathbb{Z}_p$.
Therefore, to establish \eqref{for_g_1} it suffices to prove
\begin{equation}\label{for_g_3}
  \sum_{\substack{k_1, k_2, k_3, k_4=0\\ \sum k_i  \, = \, p-1}}^{p-1} 
  g_1(p, k_1,k_2,k_3,k_4) 
  \equiv 3
  \pmod{p}.
\end{equation}
Let $k_4 \assign p-1 - k_1 -k_2 -k_3$ and consider, for $1 \leq i \leq 3$,
\begin{align*}
P(k_i)
& \assign \frac{d}{dk_i} \left[k_i \, (k_i+1)_{m}^2 \; (k_4+1)_{m}^2 \right] = b_{i,0} + \sum^{4m}_{t=1} b_{i,t} k_i^t\\
&= k_i \, (k_i+1)_{m}^2 \; \frac{d}{dk_i} \left[(k_4+1)_{m}^2 \right] + 
(k_4+1)_{m}^2 \; \frac{d}{dk_i} \left[  k_i \, (k_i+1)_{m}^2 \right].
\end{align*}
Now,
\begin{align*}
  \frac{d}{dk_i} \left[ k_i \, (k_i+1)_{m}^2 \right] 
  \notag&= (k_i+1)_{m}^2 + k_i \, \frac{d}{dk_i} \left[ (k_i+1)_{m}^2 \right] \\[9pt]
  \notag&= (k_i+1)_{m}^2 + k_i  \left[ 2 \, (k_i+1)_m \, \frac{d}{dk_i}  (k_i+1)_m \right] \\[9pt]
  \notag&= (k_i+1)_{m}^2 + k_i  \left[ 2 \, (k_i+1)_m \sum_{s=1}^{m} \prod_{\substack{r=1 \\ r \neq s}}^{m} (k_i+r)\right] \\[9pt]
  &=  (k_i+1)_{m}^2  \left[ 1 + 2 k_i   \sum_{s=1}^{m}\frac{1}{(k_i+s)}\right] \\[9pt]
  \notag&= (k_i+1)_{m}^2  \left[ 1 + 2 k_i \left( H_{m+k_i} - H_{k_i} \right)\right],
\end{align*}
and, similarly,
\begin{align*}
  \frac{d}{dk_i} \left[(k_4+1)_{m}^2 \right] 
  & = 2 \, (k_4+1)_{m} \, \frac{d}{dk_i} \left[(k_4+1)_{m} \right] \\
  & = 2 \, (k_4+1)_{m} \, \frac{d}{dk_i} \left[(p-1-k_1-k_2- k_3 +1)_{m} \right] \\
  & = - 2 \, (k_4+1)_{m} \,\sum_{s=1}^{m} \prod_{\substack{r=1 \\ r \neq s}}^{m} (p-1-k_1-k_2-k_3+r) \\
  & = - 2 (k_4+1)_{m}^2 \left( H_{m+k_4} - H_{k_4} \right).
\end{align*}
So,
\begin{equation}\label{for_pk}
  P(k_i) 
  = (k_i+1)_{m}^2 \; (k_4+1)_{m}^2 \left[ 1 + 2 k_i \left( H_{m+k_i} - H_{k_i} \right)
  -2 k_i \left( H_{m+k_4} - H_{k_4} \right) \right].
\end{equation}
Noting \eqref{exp_sums}, we see that
\begin{equation*}
  \sum_{k_1=0}^{p-1} P(k_i)
  =  p \, b_{i,0} + \sum_{k_1=1}^{p-1} \sum^{4m}_{t=1} b_{i,t} k_i^t
  \equiv  \sum^{4m}_{t=1} b_{i,t} \sum_{k_1=1}^{p-1}  {k_i}^t
  \equiv -b_{i,p-1} -b_{i,2(p-1)} 
  \pmod{p}.
\end{equation*}
By definition of $P(k_i)$, we have that
\begin{equation*}
  (k_i+1)_{m}^2 \; (k_4+1)_{m}^2 = b_{i,0} + \sum_{t=1}^{4m} \frac{b_{i,t}}{t+1} k_i^t \;,
\end{equation*}
which is monic with integer coefficients. Thus, $p\mid b_{i,p-1}$ and $b_{i,2(p-1)}=2p-1$. Therefore,
\begin{equation}\label{for_sumpk}
  \sum_{k_i=0}^{p-1} P(k_i)  \equiv -b_{i,p-1} -b_{i,2(p-1)}  \equiv 1 \pmod{p}.
\end{equation}
Considering \eqref{def_g1}, accounting for \eqref{for_pk} and \eqref{for_sumpk}, and applying Lemma~\ref{lem_BinHar2}, we get that
\begin{align*}\label{for_g_5}
  \sum_{\substack{k_1, k_2, k_3, k_4=0\\ \sum k_i  \, = \, p-1}}^{p-1} 
  g_1(p, k_1,k_2,k_3,k_4) 
  &=
  \sum_{k_1, k_2, k_3=0}^{p-1} 
  \sum_{i=1}^{3} 
  \left[
  \prod_{\substack{j=1\\j \neq i}}^{3}
  {(k_j+1)}_{m}^2
  \right]
  P(k_i)\\
  &=
  \sum_{i=1}^{3} 
  \sum_{\substack{k_j=0\\j \neq i}}^{p-1}
  \left[
  \prod_{\substack{j=1\\j \neq i}}^{3}
  {(k_j+1)}_{m}^2
  \right]
  \sum_{k_i=0}^{p-1} 
  P(k_i)\\
  & \equiv
  \sum_{i=1}^{3} 
  \sum_{\substack{k_j=0\\j \neq i}}^{p-1}
  \left[
  \prod_{\substack{j=1\\j \neq i}}^{3}
  {(k_j+1)}_{m}^2
  \right]
  \equiv 
  \sum_{i=1}^{3} 1
  = 3
  \pmod{p}.
\end{align*}
This establishes \eqref{for_g_3}, which in turn establishes \eqref{for_g_1}. Now, accounting for \eqref{for_g_1} in \eqref{for_dp4}, we see that 
\begin{equation*}
  \gamma(p) 
  \equiv 
  \sum_{\substack{k_1, k_2, k_3, k_4=0\\ \sum k_i  \, = \, p-1}}^{\frac{p-1}{2}} \,
  \prod_{i=1}^{4}
  \binom{\frac{p-1}{2}}{k_i} \binom{\frac{p-1}{2}+k_i}{k_i}
  \pmod{p^2}.
\end{equation*}
Applying Lemma \ref{lem_BinHar3} completes the proof.
\end{proof}

\section*{Acknowledgements}
The first and third authors are supported by a grant from the Simons Foundation (\#353329, Dermot McCarthy; \#514645, Armin Straub). The second author would like to thank both the Institut des Hautes \'Etudes Scientifiques and the Max-Planck-Institut f\"ur Mathematik for their support and Francis Brown for his encouragement during the initial stages of this project. He also thanks the organizers of the workshop ``Hypergeometric motives and Calabi-Yau differential equations", January 8--27, 2017 at the MATRIX institute, The University of Melbourne at Creswick, for the opportunity to discuss these results. The authors are grateful to Cl\'ement Dupont for several insightful and helpful comments on cellular integrals.

\end{document}